\documentclass[sn-mathphys,Numbered]{sn-jnl}

\usepackage{graphicx}%
\usepackage{multirow}%
\usepackage{amsmath,amssymb,amsfonts,mathtools}%
\usepackage{amsthm}%
\usepackage{mathrsfs}%
\usepackage[title]{appendix}%
\usepackage[dvipsnames]{xcolor}%
\usepackage{textcomp}%
\usepackage{manyfoot}%
\usepackage{booktabs}%
\usepackage{algorithm}%
\usepackage{algorithmicx}%
\usepackage{algpseudocode}%
\usepackage{listings}%
\usepackage{thm-restate}
\usepackage{changepage}

\theoremstyle{theorem}%
\newtheorem{theorem}{Theorem}[section]

\newtheorem{lemma}[theorem]{Lemma}
\newtheorem{corollary}[theorem]{Corollary}

\theoremstyle{example}%
\newtheorem{example}[theorem]{Example}
\newtheorem{remark}[theorem]{Remark}

\theoremstyle{definition}%
\newtheorem{definition}[theorem]{Definition}%

\newcommand{\f}[1]{\mathcal{#1}}
\newcommand{\syn}[2]{\operatorname{Syn}\left(#1,#2\right)}
\newcommand{\thi}[2]{\operatorname{Thick}\left(#1,#2\right)}
\newcommand{\ps}[2]{\operatorname{PS}\left(#1,#2\right)}
\newcommand{\fs}[1]{\f{P}_f\left(#1\right)}

\begin{document}

\title[A characterization of piecewise $\f{F}$-syndetic sets]{A characterization of piecewise $\f{F}$-syndetic sets}

\author{\fnm{Conner} \sur{Griffin}}\email{cgrffn11@memphis.edu}

\affil{\orgdiv{Department of Mathematical Sciences}, \orgname{The University of Memphis}, \orgaddress{\city{Memphis}, \state{TN}, \postcode{38152}, \country{USA}}}

\abstract{Some filter relative notions of size, $\left( \f{F},\f{G}\right) $-syndeticity and piecewise $\f{F} $-syndeticity, were defined and applied with clarity and focus by Shuungula, Zelenyuk and Zelenyuk \cite{SZZ}. These notions are generalizations of the well studied notions of syndeticity and piecewise syndeticity. Since then, there has been an effort to develop the theory around the algebraic structure of the Stone-\v{C}ech compactification so that it encompasses these new generalizations. We prove one direction of a characterization of piecewise $\f{F}$-syndetic sets. This completes the characterization, as the other direction was proved by Christopherson and Johnson \cite[Theorem 4.4]{CJ}.}

\keywords{Stone-\v{C}ech Compactification, syndetic sets, thick sets, piecewise syndetic sets, ultrafilters}


\pacs[MSC Classification]{54D80, 22A15}

\maketitle

\tableofcontents
\section{Notions of size in a semigroup}
We will use the following notation throughout this paper.
\begin{enumerate}
	\item We take $\mathbb{N} := \{1, 2, 3, \dots\}$
	\item For a set $S$, $\fs{S}:=\{A \subseteq S: \ \varnothing \ne A, \ \left|A\right|< \infty \}.$
	\item For a semigroup $S$, $A \subseteq S$ and $h \in S$, $h^{-1}A := \{x: hx \in A\}.$
\end{enumerate}

There are several different notions of size in a semigroup that have applications in combinatorial number theory.

\begin{definition}[Syndetic, thick, piecewise syndetic]
	Let $S$ be a semigroup.
	\begin{enumerate}
		\item A set $A \subseteq S$ is \emph{syndetic} if and only if there exists some $H \in \fs{S}$ such that $S = \bigcup_{h \in H}h^{-1}A.$
		\item A set $A \subseteq S$ is \emph{thick} if and only if for all $H \in \fs{S}$ we have $\bigcap_{h \in H}h^{-1}A \ne \varnothing.$
		\item A set $A \subseteq S$ is \emph{piecewise syndetic} if and only if there is some $H \in \fs{S}$ such that \[\left\{x^{-1}\left( \bigcup_{h \in H}h^{-1}A \right): \quad x\in S \right\}\] has the finite intersection property. 
	\end{enumerate}
\end{definition}
In the semigroup $\left(\mathbb{N},+\right)$, syndetic sets are the sets with bounded gaps and thick sets are the sets with arbitrarily long sequences of consecutive natural numbers. In this case -- and in the case of the generalization to arbitrary semigroups -- remark \ref{du} should be clear.
\begin{remark}\label{du}
	Let $S$ be a semigroup and let $A \subseteq S$.
	\begin{enumerate}
		\item $A$ is thick if and only if $A^c$ is not syndetic.
		\item $A$ is syndetic if and only if $A^c$ is not thick.
		\item For any syndetic set $A$ and any thick set $B$, $A \cap B \ne \varnothing.$
	\end{enumerate} 
\end{remark}
The following characterization of piecewise syndetic is important.
\begin{restatable}{theorem}{abc}\label{abc}
	A set is piecewise syndetic if and only if it is the intersection of a syndetic set and a thick set.
\end{restatable}
We will prove one direction of this after introducing the Stone-\v{C}ech compactification of a semigroup. We will present a new proof which uses the relationship between piecewise syndetic sets and minimal idempotent ultrafilters.

The following theorem shows the connection between piecewise syndeticity and combinatorial number theory. It can be used to prove Van der Waerden's theorem. Indeed, if $\bigcup_{i=1}^kC_i=\mathbb{N}$ is a finite partition of the natural numbers then there is an $1\le i \le k$ such that $C_i$ is piecewise syndetic.
\begin{theorem} \label{vdw}
	Let $A$ be a piecewise syndetic set in $\mathbb{N}$ and let $\ell\in \mathbb{N}.$ There exist $a,d \in \mathbb{N}$ such that $\{a, a+d, a+2d, \dots, a+\ell d\} \subseteq A.$
\end{theorem}
For a proof of this theorem see \cite[Theorem 14.1]{HiSt1}. The first algebraic proof of Theorem \ref{vdw} along these lines is due to Furstenberg and Katznelson. \cite{FK1} These proofs are independent of Van der Waerden's Theorem.

This gives a combinatorial property of piecewise syndetic sets, but it does not characterize them. That is, there are sets which contain arbitrarily long arithmetic progressions which are not piecewise syndetic.
\begin{example}
	The set $\cup_{n=1}^{\infty}\{n!, n!+n, n!+2n,\dots, n!+n^2\}$ is not piecewise syndetic.
\end{example}

There are many other notions of size that have found applications in combinatorics. It is a well studied topic. \cite{HMS, BHM, HiSt1, Hi}

\section{The Stone-\v{C}ech compactification of a semigroup}
There are several set theoretic notions that we need in order to define the Stone-\v{C}ech compactification.
\begin{definition}[Filter, prime, ultrafilter]
	Let $S$ be a nonempty set and let $\f{F} \subseteq \mathcal{P}\left(S\right).$
	\begin{itemize}
		\item We call $\f{F}$ a \emph{filter on $S$} if it satisfies the following:
		\begin{enumerate}
			\item $\varnothing \not\in \f{F}$.
			\item (Upward closed) if $A \in \f{F}$ and $A \subseteq B \subseteq S$ then $B \in \f{F}.$
			\item (Downward directed) if $A \in \f{F}$ and $B\in \f{F}$ then $A \cap B \in \f{F}.$
		\end{enumerate}
		\item We call $\f{F}$ \emph{prime} if it has the following property:
		\begin{enumerate}
			\item[4.] if $A \cup B \in \f{F}$ then $A \in \f{F}$ or $B \in \f{F}.$
		\end{enumerate}
		\item We call $\f{F}$ an \emph{ultrafilter on $S$} if and only if $\f{F}$ is a prime filter.
	\end{itemize}
\end{definition}

There are many equivalent ways of defining ultrafilters. Alternatively, $\f{F}$ is an ultrafilter if and only if it is a filter which is not properly contained in any other filter. In lattice and order therory it is common to refer to the 4th property above as \emph{prime}. It may also be called partition regularity. We use the following specific case of primality often: for any ultrafilter, $p,$ on a set $S$ and for any set $A \subseteq S$, either $A \in p$ or $A^c \in p.$

\begin{definition}[The Stone-\v{C}ech compactification]
	Let $S$ be a discrete topological space. Then the Stone-\v{C}ech compactification of $S$ is the set \[\beta S:=\{p: \ p \ \textrm{is an ultrafilter on} \ S\}\] together with the topology generated by the closed base \[\{\{p \in \beta S: \ A \in p\}: \ \varnothing \ne A \subseteq S\}.\]
\end{definition}

\begin{definition}[Principal ultrafilters]
	For any $x \in S$, the ultrafilter $\{A \subseteq S: \ x \in A\}$ is called the \emph{principal ultrafilter} associated with $x.$
\end{definition}
\begin{remark}
	It is a typical abuse of notation to identify $x\in S$ with the ultrafilter $\{ A \subseteq S: \ x \in A \}$ and to refer to the set of all principal ultrafilters as $S$. For $A \subseteq S$, the $\beta S$ closure of $A$ is $\overline{A} = \{ p \in \beta S: \ A \in p\}$.
\end{remark}

\begin{lemma}
	Let $S$ be a set and let $\f{A} \subseteq \f{P}\left(S\right)$ have the finite intersection property. Then there is a filter which contains $\f{A}.$
\end{lemma}

\begin{theorem}
	Let $\f{F}$ be a filter on a set $S$. Then there exists an ultrafilter in $\beta S$ which contains $\f{F}.$
\end{theorem}

From the two preceding theorems, we immediately have the following corollary.
\begin{corollary}
	Let $S$ be a set and let $\f{A} \subsetneq \f{P}\left(S\right)$ have the finite intersection property. Then there is an ultrafilter which contains $\f{A}.$
\end{corollary}
\begin{definition}[Filter closure]
	Let $\f{F}$ be a filter on a set $S$. Then the \emph{filter closure} of $\f{F}$ is $\overline{ \f{F} }:=\{p\in \beta S: \f{F} \subseteq p\}.$
\end{definition}

An equivalent definition of the filter closure is $\overline{ \f{F} }:= \bigcap_{V \in \f{F}}\overline{V}.$ Be aware that, while we use the same notation, $\overline{ \f{F} }$ is \emph{not} the topological closure of $\f{F}.$

If $T$ is a closed nonempty subset of $\beta S$, then there is a filter $\f{F}$ with $T = \overline{ \f{F} }.$ If $\f{F} = \{S\}$ then $\overline{\f{F}}= \beta S$.
The ultrafilters which are not principal are called \emph{non-principal ultrafilters}. There are many non-principal ultrafilters and many ways to characterize them. They are the ultrafilters which can only be defined using the axiom of choice. They are also the ultrafilters which have no countable filter base. They are also the ultrafilters with no finite sets. If $\f{F}=\{A \subseteq S: \ \left| A^c \right|<\infty\}$ then the non-principal ultrafilters are the ultrafilters in $\overline{ \f{F} }$. The filter $\{A \subseteq S: \ \left|A^c\right| < \infty\}$ is known as the cofinite or Fr\'{e}chet filter. See \cite{CN} for more on ultrafilters.

In the case when $S$ is a semigroup, there is an associative binary operation on $\beta S$ which extends the semigroup operation of $S$.
\begin{definition}[Filter product]
	Let $\f{F}$ and $\f{G}$ be filters on a set $S$. Define $$\f{F}\cdot \f{G}:=\{A \in \f{P}\left(S\right): \ \{x \in S: \ x^{-1}A \in \f{G}\}\in \f{F}\}.$$
\end{definition}
For any two filters, $\f{F}$ and $\f{G},$ on a semigroup $S$, $\f{F} \cdot \f{G}$ is a filter and the operation is associative. As well, $\f{F} \cdot \f{G}$ is an ultrafilter  if $\f{F}$ and $\f{G}$ are ultrafilters. This product restricted to $\beta S$ extends the semigroup operations of $S$ meaning that $x \cdot y = xy$ for principal ultrafilters $x,y \in S.$ (The semigroup operation is represented by juxtaposition.)
Under this operation, $\beta S$ is a Hausdorff compact right topological semigroup, where right topological means that for all $q \in \beta S$ the map $p \mapsto p \cdot q$ is continuous.

\begin{theorem}[Ellis' Theorem]\cite{ellis}
	\label{EN}
	Let $S$ be a Hausdorff compact right topological semigroup. Then $S$ contains an idempotent element.
\end{theorem}

\begin{definition}[Ideal of a semigroup]
	Let $S$ be a semigroup. Then,
	\begin{itemize}
		\item $L \subseteq S$ is a \emph{left ideal} of $S$ if and only if $L \ne \varnothing$ and $SL \subseteq L.$
		\item $R \subseteq S$ is a \emph{right ideal} of $S$ if and only if $R \ne \varnothing$ and $RS \subseteq R.$
		\item $I \subseteq S$ is an \emph{ideal} of $S$ if it is both a left and right ideal.
	\end{itemize}
\end{definition}
\begin{definition}[Minimal ideals]
	Let $S$ be a semigroup. Then,
	\begin{itemize}
		\item ($L \subseteq S$ is a \emph{minimal left ideal} of $S$) if and only if ($L$ is a left ideal of $S$ and if $J$ is a left ideal of $S$ with $J\subseteq L$ then $J = L$).
		\item ($R \subseteq S$ is a \emph{minimal right ideal} of $S$) if and only if ($R$ is a right ideal of $S$ and if $J$ is a right ideal of $S$ with $J \subseteq R$ then $J =R$).
		\item ($I \subseteq S$ is a \emph{minimal ideal} of $S$) if and only if ($I$ is an ideal and if $J$ is an ideal of $S$ with $J \subseteq I$ then $J=I$).
	\end{itemize}
\end{definition}

\begin{theorem} \cite[Lemma 1.49]{HiSt1}
	Let $S$ be a semigroup and let $K$ be an ideal of $S.$ If $K$ is minimal in $\{J: \ J \ \textrm{is an ideal of} \ S\}$ and $I$ is an ideal of $S$, then $K \subseteq I.$
\end{theorem}

This allows us to conclude that if $S$ has a minimal ideal it must be unique. We use $K\left(S\right)$ to denote the minimal ideal of $S.$ The next result gives us a simple condition which guarantees the existence of a minimal ideal.
\begin{theorem} \label{uli} \cite[Theorem 1.51]{HiSt1}
	Let $S$ be a semigroup. If $S$ has a minimal left ideal, then $K\left(S\right)$ exists and $K\left(S\right) = \cup\{L: \ L \ \textrm{is a minimal left ideal of} \ S\}.$
\end{theorem}

\begin{theorem}\cite[Corollary 2.6]{HiSt1}\label{crt}
	Let $S$ be a compact right topological semigroup. Then every left ideal of $S$ contains a minimal left ideal. Minimal left ideals are closed, and each minimal left ideal has an idempotent.
\end{theorem}

We can conclude that for any semigroup, $S,$ $K\left(\beta S\right)$ exists and contains an idempotent. In addition, if $\f{F}$, a filter on a semigroup $S,$ has the property that $\overline{ \f{F} }$ is a subsemigroup of $\beta S,$ then $K\left(\overline{ \f{F} }\right)$ exists and contains an idempotent. We call idempotents in $K\left(\beta S\right)$ (or $K\left( \overline{ \f{F} }\right) $) minimal idempotents and denote the set of all minimal idempotents by $E\left( K\left( \beta S\right) \right)$ (or $E\left( K\left( \overline{\f{F}}\right) \right)$).

We can characterize piecewise syndetic sets by the way they relate to minimal idempotents of $\beta S$.
\begin{theorem}\cite[Theorem 4.39]{HiSt1}
	Let $S$ be a semigroup and let $p \in \beta S$. The following statements are equivalent:
	\begin{enumerate}
		\item $p \in K\left( \beta S\right).$
		\item For all $A \in p,$ $\{x \in S: \ x^{-1}A \in p\}$ is syndetic.
		\item For all $q \in \beta S$, $p \in \beta S \cdot q \cdot p.$
	\end{enumerate}
\end{theorem}
\begin{theorem}\cite[Theorem 4.40]{HiSt1}
	Let $S$ be a semigroup and let $A \subseteq S.$ Then $\overline{A} \cap K\left( \beta S\right) \ne \varnothing$ if and only if $A$ is piecewise syndetic.
\end{theorem}
\begin{remark}
	We use $A'\left( p\right)$ to denote the set $\{x \in S: \ x^{-1}A \in p\}$.
\end{remark}
\begin{theorem}\cite[Theorem 4.43]{HiSt1}\label{aps}
	Let $S$ be an infinite semigroup and let $A \subseteq S$. The following statements are equivalent:
	\begin{enumerate}
		\item $A$ is piecewise syndetic
		\item There is a minimal idempotent, $e$, such that $\{x \in S: \ x^{-1}A \in e\}$ is syndetic.
		\item There is a minimal idempotent, $e,$ and some $x \in S$ such that $x^{-1}A \in e$. 
	\end{enumerate}
\end{theorem}

This theorem is slightly different from the referenced theorem found in \cite{HiSt1} and it is not clear that they are equivalent. Because of this, we prove that $(1) \Longrightarrow (2).$ That $(2) \Longrightarrow (3)$ is trivial and that $(3) \Longrightarrow (1)$ can be found in \cite{HiSt1}.

\begin{proof}
	Suppose $A$ is piecewise syndetic. Then there exists $p \in K\left(\beta S\right)$ such that $A \in p.$ Since $\beta S \cdot p$ is a left ideal, it contains a minimal left ideal $L.$ There is an idempotent $e \in L$ with $L = \beta S \cdot e.$ We will show that $\overline{A'\left(e\right)}$ meets every left ideal of $\beta S$ such that by \cite[Theorem 4.48]{HiSt1}, $A'\left(e\right)$ is syndetic. Let $u \in \beta S$. Then we have $u \cdot e \in L \subseteq \beta S \cdot p.$ This allows us to pick $q \in \beta S$ such that $u \cdot e = q\cdot p.$ By \cite[Theorem 4.39]{HiSt1}, there is an $r \in \beta S$ such that $p = r\cdot q \cdot p = r \cdot u \cdot e.$ Thus $A \in p = r \cdot u \cdot e$ and then $A'\left(e\right) \in r\cdot u.$ Hence, for all $u\in \beta S,$ $\overline{A'\left(e\right)} \bigcap \beta S \cdot u \ne \varnothing.$ In particular, $\overline{A'\left(e\right)}$ meets every minimal left ideal of $\beta S.$
\end{proof}

\begin{lemma}\label{il}
	Let $p \in \beta S$ and $A \subseteq S.$ Then $\left( A'\left(p\right)\right) ^c = \left( A^c\right) '\left(p\right).$
\end{lemma}
\begin{proof}
	\[\left( A'\left(p\right)\right) ^c = \{x \in S: \ x^{-1}A \notin p\} = \{x \in S: \ x^{-1}A^c \in p\} = \left( A^c\right) '\left(p\right)\]
\end{proof}
\begin{lemma}\label{ids}
	For any idempotent ultrafilter $e,$ and any $A \subseteq S$ and for all $x \in A'\left( e\right)$ we have $x^{-1}A'\left( e\right) \in e.$
\end{lemma}
\begin{proof}
	Take $x \in A'\left( e\right) = \{x \in S: \ x^{-1}A \in e\}.$ Then, $x^{-1}A \in e = e\cdot e.$ Equivalently, $x^{-1}A'\left( e\right)  =\{y \in S: \ y^{-1}x^{-1}A 	\in e\} \in e$
\end{proof}
We present a new proof of one direction of Theorem \ref{abc}. The proof of the main result, Theorem \ref{cpfs}, will follow this proof with some modifications.

\abc*

\begin{proof}
	Suppose $A$ is piecewise syndetic. Take, by Theorem \ref{aps}, a minimal idempotent ultrafilter, $e,$ such that $A'\left( e\right) $ is syndetic. Then $A \cup A'\left( e\right) $ is also syndetic because supersets of syndetic sets are syndetic.

	By Lemma \ref{ids}, for all $x \in \left( A^{c}\right)'\left( e\right)$, $x^{-1}\left( A^{c}\right)'\left( e\right) \in e. $ Also, for all $x \in A'\left( e\right),$ $x^{-1}A \in e.$ Altogether,
	
	\[\left(\forall x \in S\right) \ \left(x^{-1}\left( A \cup \left( A^{c}\right)'\left( e\right) \right) \in e\right). \]
	
	Hence, $\{x^{-1}\left( A \cup \left( A^{c}\right)'\left( e\right) \right): \ x\in S\}$ has the finite intersection property or equivalently $A \cup \left( A^{c}\right)'\left( e\right)$ is thick. The final point
	
	\[A = \left( A \cup A'\left( e\right)\right)  \cap \left(A \cup \left( A^{c}\right)'\left( e\right)\right). \]
	
	For the converse see \cite[Theorem 4.49]{HiSt1}.
\end{proof}

\section{Algebraic characterizations of filter relative notions of size}
In \cite{CJ}, Christopherson and Johnson use the following useful term which we adopt.
\begin{definition}[Mesh]
	Let $S$ be a nonempty set and let $\f{F} \subseteq \f{P}\left(S\right).$ The \emph{mesh of $\f{F}$} is $\f{F}^* = \{A \subseteq S: \ A^c \notin \f{F}\}.$
\end{definition}

When $\f{F}$ is a filter there is an equivalent definition: $\f{F}^* = \{A \subseteq S: \ B \cap A \ne \varnothing \ \textrm{for all} \ B \in \f{F}\}.$ Again, when $\f{F}$ is a filter, it is useful to think of $\f{F}^*$ as the collection of all sets which may extend $\f{F}$ in the sense that for any $A \in \f{F}^*$, $\{A\} \cup \f{F}$ has the finite intersection property and hence must be contained in an ultrafilter. As a consequence of this perspective, one can see that if $\f{F}$ and $\f{G}$ are filters with $\f{F} \subseteq \f{G}$ then $\f{G}^* \subseteq \f{F}^*.$

In \cite{SZZ}, Shuungula, Zelenyuk and Zelenyuk characterized the closure of the minimal ideal of an aribitrary subsemigroup of $\beta S.$ Crucial to this paper are the filter relative notions of size, $\left( \f{F}, \f{G}\right) $-syndeticity and piecewise $\f{F}$-syndeticity. The notion of $\left( \f{F}, \f{G}\right)$-thick comes from \cite{CJ}.

\begin{definition}[Relatively syndetic, thick, piecwise syndetic]
	Let $S$ be a semigroup, $A \subseteq S,$ and $\f{F}$, $\f{G}$ be filters on $S.$
	\begin{itemize}
		\item We call $A$ \emph{$\left(\f{F},\f{G}\right)$-syndetic} if and only if for every $V \in \f{F}$ there exists $H_V \in \f{P}_f\left(V\right)$ such that $\bigcup_{h \in H_V}h^{-1}A \in \f{G}.$ As well, we use $\syn{\f{F}}{\f{G}}$ to represent the collection of all $\left(\f{F},\f{G}\right)$-syndetic sets.
		\item We call $A$ \emph{$\left(\f{F},\f{G}\right)$-thick} if and only if there exists a $V \in \f{F}$ such that for all $H \in \mathcal{P}_f\left(V\right)$ we have $\bigcap_{h \in H}h^{-1}A \in \f{G}^*.$ As well, we use $\thi{\f{F}}{\f{G}}$ to represent the collection of all $\left(\f{F},\f{G}\right)$-thick sets.
		\item We call $A$ \emph{piecewise $\f{F}$-syndetic} if and only if for all $V \in \f{F}$ there exists $H_V \in \fs{V}$ and $W_{V} \in \f{F}$ such that
		\[\left\{\left( x^{-1}\bigcup_{h \in H_V}h^{-1}A\right)  \cap V: \ V \in \f{F}, \ x \in W_V\right\}\] has the finite intersection property.
	\end{itemize}
\end{definition}

When $\f{G} =\f{F}$ we simply say $\f{F}$-syndetic and $\f{F}$-thick. When $\f{F} = \{S\}$ then these notions are equivalent to the usual notions of size. Piecewise $\f{F}$-syndeticity has been used to achieve combinatorial results in \cite{JR17}.

The next three results are from \cite{CJ}. In that paper, Christopherson and Johnson have presented many results expanding the theory around these relative notions of size. Interestingly, a set $A \subseteq S$ is piecewise $\f{F}$-syndetic if and only if $A \in \syn{\f{F}}{\thi{\f{F}}{q}}$ for some $q \in \overline{ \f{F} }$. While it is not necessarily the case that $\thi{\f{F}}{\f{G}}$ is a filter for arbitrary $\f{F}$ and $\f{G}$, we will see that $\thi{\f{F}}{q} = \f{F} \cdot q$ for any filter $\f{F}$ and ultrafilter $q.$  From this interpretation of piecewise $\f{F}$-syndeticity and from the theorems below, Christopherson and Johnson proved that any set which is the intersection of an $\f{F}$-syndetic set and an $\f{F}$-thick set is piecewise $\f{F}$-syndetic. The converse was left as an open question. \cite[Question 4.6]{CJ}
\begin{theorem}\cite[Proposition 3.2]{CJ}\label{dual}
	Let $\f{F}, \f{G}$ be filters on a semigroup $S.$
	
	$$\syn{\f{F}}{\f{G}} = \left( \thi{\f{F}}{\f{G}}\right)^*.$$
\end{theorem}

\begin{theorem}\cite[Proposition 2.5 (b)]{CJ}
	Let $\f{F}$ be a filter on a set $S.$ Then $\f{F} = \left(\f{F}^*\right)^*.$
\end{theorem}
This establishes a duality between the filter relative notions of syndeticity and thickness that is well known in the non-relative case. The next results generalize \cite[Theorem 4.48]{HiSt1}.
\begin{theorem}\label{sych}\cite[Lemma 3.9 (a')]{CJ}
	Let $\f{F},\f{G}$ be filters on a semigroup $S.$ Then $$\syn{\f{F}}{\f{G}} = \{A\subseteq S: \ \textrm{for every} \ q\in\overline{\f{G}} \ \textrm{we have} \ \overline{\f{F}} \cdot q \cap \overline{A} \ne \varnothing \}.$$
\end{theorem}
\begin{corollary}\cite[Lemma 3.9 (a)]{CJ}\label{thch}
	Let $\f{F},\f{G}$ be filters on a semigroup $S.$ Then $$\thi{\f{F}}{\f{G}} = \{A\subseteq S: \ \textrm{there exists} \ q\in\overline{\f{G}} \ \textrm{such that} \ \overline{\f{F}} \cdot q \subseteq \overline{A} \}.$$
\end{corollary}
\begin{theorem}\label{elim}
	Let $\f{F}$ be a filter on a set $S$. Then
	\begin{flalign}\bigcap_{p \in \overline{\f{F}}}p 
		&= \f{F} \\ \bigcup_{p \in \overline{\f{F}}}p &= \f{F}^*\end{flalign}
\end{theorem}
\begin{proof}
	If $\f{F}$ is an ultrafilter then both $\left(1\right)$ and $\left(2\right)$ are clear because in this case $\overline{\f{F}} = \{\f{F}\}$.
	
	For $\left(1\right)$, it is clear that $\f{F} \subseteq \bigcap_{p \in \overline{\f{F}}}p.$ For the reverse inclusion we will go by contradiction. Suppose that there is an $A \in \left( \bigcap_{p \in \overline{\f{F}}}p\right)  \setminus \f{F}.$ It cannot be that $A$ is a superset of any set in $\f{F}$ as $A$ would then be in $\f{F}$. Therefore, $B \setminus A \ne \varnothing$ for all $B \in \f{F}.$ But this means that $A^c \cap B \ne \varnothing$ for all $B \in \f{F}.$ By the preceding line we may conclude that there is an ultrafilter, $p,$ in $\overline{\f{F}}$ which contains $A^c.$ Since $A$ is an element of every ultrafilter in $\overline{\f{F}},$ we have $A$ and $A^c$ in $p.$ Contradiction.
	
	For $\left(2\right),$ if $A \in p \in \overline{\f{F}}$ then $A \cap B \ne \varnothing$ for all $B \in \f{F}.$ Therefore, $A \in \f{F}^*.$ So $\bigcup_{p \in \overline{\f{F}}}p \subseteq \f{F}^*.$ Conversely, if $A \in \f{F}^*$ then $\f{F}\cup \{A\}$ has the finite intersection property and hence must be contained in an ultrafilter.
\end{proof}

\section{Results}
We first build on the characterizations in Theorem \ref{sych} to achieve computationally useful characterizations of $\syn{\f{F}}{\f{G}},$ $\thi{\f{F}}{\f{G}}$ and $\bigcup_{q\in \overline{ \f{F} }}\syn{\f{F}}{\thi{\f{F}}{q}}$.
In order to do this, we first note that we may extend the filter product to collections of sets that are simply upward closed. Christopherson and Johnson refer to such sets as $\emph{stacks}.$ If $\f{F}$ and $\f{G}$ are stacks then we define the product as we have been:
\[\f{F} \cdot \f{G} := \{A: \ \{x: \ x^{-1}A \in \f{G}\}\in \f{F} \}.\]
When we do this, $\f{F}\cdot \f{G}$ is a stack and the operation is associative. The mesh of a filter is a stack. For $\f{A}$, an arbitrary collection of filters on a set $S$, $\bigcap_{\f{F}\in\f{A}}\f{F}$ is a filter and $\bigcup_{\f{F} \in \f{A}}\f{F}$ is a stack.
\begin{theorem}\label{firstposition}
	Let $S$ be a semigroup and $\f{A} \subseteq \beta S$. Then \(\forall q \in \beta S\)
	\[\bigcap_{p \in \f{A}}\left(p\cdot q\right) = \left(\bigcap_{p \in \f{A}}p\right)\cdot q \quad \textrm{and} \quad \bigcup_{p \in \f{A}}\left(p\cdot q\right) = \left(\bigcup_{p \in \f{A}}p\right)\cdot q.\]
\end{theorem}
\begin{proof}
	Let $q \in \beta S.$
	\begin{flalign*}
		\forall p \in \f{A}, \quad A \in p\cdot q \quad  &\iff \quad\forall p \in \f{A}, \quad \{x: \ x^{-1}A \in q \} \in p \\
		& \iff \quad \{x: \ x^{-1}A \in q \} \in \bigcap_{p \in \f{A}}p \\
		& \iff \quad A \in \left( \bigcap_{p \in \f{A}}p \right) \cdot q.
	\end{flalign*}
	Likewise,
	\begin{flalign*}
		\exists p \in \f{A}, \quad A \in p\cdot q \quad &\iff \quad \exists p \in \f{A}, \quad \{x: \ x^{-1}A \in q \} \in p \\
		& \iff \quad \{x: \ x^{-1}A \in q \} \in \bigcup_{p \in \f{A}}p \\
		& \iff \quad A \in \left( \bigcup_{p \in\f{A}}p \right) \cdot q.
	\end{flalign*}
\end{proof}

It is the case that, for a filter $\f{F}$ and ultrafilter $p$, if $p \subseteq \f{F}^*$ then $p \in \overline{ \f{F} }.$ Indeed, since $p \subseteq \f{F}^*$, $p \cup \f{F}$ has the finite intersection property. Hence, there is an ultrafilter, $q$, with $p \cup \f{F} \subseteq q.$ Then $q \in \overline{\f{F}}$ and $q=p.$ We adapt this proof to the following useful lemma.

\begin{lemma}\label{enlight}
	Let $\f{F}$ be a filter on a semigroup $S$ and let $r,p \in \beta S.$ If $r \subseteq \f{F}^* \cdot p$, then $r \in \overline{\f{F}} \cdot p$.
\end{lemma}

\begin{proof}
	Suppose $r \subseteq \f{F}^* \cdot p.$ Then for every finite $H \subseteq r$, $$\bigcap_{A \in H} A'\left(p\right) = \left(\bigcap_{A \in H} A\right)'\left(p\right)  \in \f{F}^*,$$ because $\bigcap_{A \in H}A \in r.$ Hence, $\f{F}\cup \{A'\left(p\right) : \ A \in r\}$ has the finite intersection property. There is a $q \in \beta S$ such that $\f{F}\cup \{A'\left(p\right) : \ A \in r\} \subseteq q.$ Thus $q \in \overline{\f{F}}$ and $ r \subseteq q \cdot p.$ Being ultrafilters, $r = q\cdot p.$
\end{proof}

The following result is useful. Recall that for any two filters $\f{F}$ and $\f{G}$, it may not be that $\overline{ \f{F} \cdot \f{G} } = \overline{\f{F}} \cdot \overline{\f{G}}$.

\begin{theorem} \label{compcon}
	Let $\f{F}$ be a filter on a semigroup $S.$ Then $\overline{\f{F}}\cdot q = \overline{\f{F}\cdot q}$ and $\f{F}^*\cdot q = \left(\f{F} \cdot q\right)^*$ for any $q\in \beta S$.
\end{theorem}
\begin{proof}
	We first prove the later result.
	\begin{flalign*}
		A \in \left(\f{F}\cdot q\right)^* \quad & \iff \quad A^c \notin \f{F}\cdot q = \bigcap_{p \in \overline{\f{F}}}p\cdot q\\
		& \iff \quad  \left( \exists p\in \overline{\f{F}}\right)  \ \left( A^c \notin p\cdot q\right)  \\
		& \iff \quad \left( \exists p\in \overline{\f{F}}\right)  \ \left( A \in p\cdot q\right) \\
		& \iff \quad A \in \bigcup_{p \in \overline{\f{F}}}p\cdot q = \f{F}^*\cdot q.
	\end{flalign*}
	
	For the former result, if $r\in \overline{\f{F}}\cdot q$ then there is a $p \in \overline{\f{F}}$ such that $r = p\cdot q.$ Of course $p\cdot q \supseteq \f{F}\cdot q$ which means that $p\cdot q \in \overline{\f{F}\cdot q}.$ So, $\overline{\f{F}}\cdot q \subseteq \overline{\f{F}\cdot q}.$
	
	Let $r \in \overline{\f{F}\cdot q}.$ Then $r \subseteq \left( \f{F}\cdot q\right) ^* = \f{F}^*\cdot q.$ By Lemma \ref{enlight}, $r \in \overline{\f{F}}\cdot q.$
\end{proof}

\begin{theorem}\label{sup}
	Let $\f{F}, \f{G}$ be filters on a semigroup S. Then $$\syn{\f{F}}{\f{G}} = \bigcap_{q \in \overline{\f{G}}} \bigcup_{p \in \overline{\f{F}}} \left( p \cdot q\right) $$ and $$\thi{\f{F}}{\f{G}} = \bigcup_{q \in \overline{\f{G}}}\bigcap_{p \in \overline{\f{F}}}\left( p\cdot q\right) .$$
\end{theorem}
\begin{proof}
	\begin{flalign*}
		A \in \syn{\f{F}}{\f{G}} \quad \iff& \quad \left(\forall V \in \f{F}\right) \ \left(\exists H_V \in \fs{V}\right)  \ \left(\bigcup_{h \in H_V}h^{-1}A \in \f{G}\right) \\
		\iff& \quad \neg \left[\left(\exists V \in \f{F}\right) \ \left(\forall H \in \fs{V}\right) \ \left(\bigcup_{h \in H}h^{-1}A \notin \f{G}\right)\right] \\
		\iff& \quad \neg \left[\left(\exists V \in \f{F}\right) \ \left(\forall H \in \fs{V}\right) \ \left(\bigcap_{h \in H}h^{-1}A^c \in \f{G}^*\right)\right] \\
		\iff& \quad \neg \big[\left(\exists V \in \f{F}\right) \ \left(\forall H \in \fs{V}\right) \ (\f{G} \cup \{x^{-1}A^c: x \in V\} \\
		&\quad\phantom{\neg\big[} \textrm{has the f.i.p.})\big] \\
		\iff& \quad \neg \left[\left(\exists q \in \overline{ \f{G} }\right) \ \left(\exists V \in \f{F}\right) \ \left(\forall x \in V\right) \ \left(x^{-1}A^c \in q\right)\right] \\
		\iff& \quad \neg \left[ \left(\exists q \in \overline{\f{G}}\right) \ \left(A^c \in \f{F} \cdot q\right) \right] \\
		\iff& \quad \left(\forall q \in \overline{\f{G}}\right) \ \left(A^c \notin \f{F} \cdot q\right)\\
		\iff& \quad \left(\forall q \in \overline{\f{G}}\right) \ \left(A \in \left( \f{F} \cdot q\right)^*\right) \\
		\iff& \quad \left(\forall q \in \overline{\f{G}}\right) \ \left(A \in \f{F}^* \cdot q\right) \\
		\overset{\ref{elim}}{\iff}& \quad \left(\forall q \in \overline{\f{G}}\right) \ \left(A \in \left(\bigcup_{p \in \overline{\f{F}}}p\right) \cdot q\right) \\
		\overset{\ref{firstposition}}{\iff}& \quad \left(\forall q \in \overline{\f{G}}\right) \ \left(A \in \bigcup_{p \in \overline{\f{F}}}\left(p \cdot q \right)\right) \\
		\iff & \left(\forall q \in \overline{\f{G}}\right) \ \left(\exists p \in \overline{\f{F}}\right) \ \left(A \in p \cdot q\right)
	\end{flalign*}
	Thus $$\syn{\f{F}}{\f{G}} = \bigcap_{q \in \overline{\f{G}}}\bigcup_{p \in \overline{\f{F}}}\left( p\cdot q\right) .$$
	
	\noindent One can apply Theorem \ref{dual} to get \[\thi{\f{F}}{\f{G}} = \bigcup_{q \in \overline{\f{G}}}\bigcap_{ p \in \overline{ \f{F} } }\left( p\cdot q\right) .\]
	
	\noindent Indeed, 
	\begin{flalign*}
		\thi{\f{F}}{\f{G}} &= \left(\syn{\f{F}}{\f{G}}\right)^* \\
		&= \{A \subseteq S: \ A^c \not\in \syn{\f{F}}{\f{G}}\}\\
        &= \{A \subseteq S: \ \neg\left[ \left(\forall q \in \overline{\f{G}}\right) \ \left(\exists p \in \overline{\f{F}}\right) \ \left(A^c \in p\cdot q\right)\right]\}\\
		&= \{A \subseteq S: \ \left(\exists q \in \overline{\f{G}}\right) \ \left(\forall p \in \overline{\f{F}}\right)  \ \left(A^c \not\in p \cdot q\right) \} \\
		&= \{A \subseteq S: \ \left( \exists q \in \overline{\f{G}}\right) \ \left(\forall p \in \overline{\f{F}}\right) \ \left(A \in p \cdot q\right) \} \\
		&= \bigcup_{q \in \overline{\f{G}}}\bigcap_{p \in \overline{\f{F}}} \left( p\cdot q\right) 
	\end{flalign*}
\end{proof}

It is well known that for a compact topological space $X$, a continuous function $f: X \to X$, and a set $A \subseteq X$ we have $\overline{f\left(A\right)} = f\left(\overline{A}\right).$ While multiplication on the right is continuous in our setting, that $\overline{\f{F}}\cdot q = \overline{\f{F}\cdot q}$ is not an immediate application of this fact. Indeed, $\f{F}$ is not a subset of $\beta S$.

We now characterize piecewise $\f{F}$-syndeticity in the same manner. The author has no proof of $\left(2\right) \iff \left(3\right)$ in Theorem \ref{eq} which does not rely on the equality $\overline{ \f{F}\cdot q } = \overline{ \f{F} }\cdot q$.

\begin{theorem}\label{eq}
	Let $\f{F}$ be a filter on a semigroup $S$ and let $A \subseteq S.$ The following are equivalent:
	\begin{enumerate}
		\item $A$ is piecewise $\f{F}$-syndetic;
		\item $A \in \bigcup_{r \in \overline{ \f{F} }}\syn{\f{F}}{\thi{\f{F}}{r}}$;
		\item $A \in \bigcup_{r \in \overline{ \f{F} }}\bigcap_{ q \in \overline{ \f{F} } }\bigcup_{p \in \overline{\f{F}}}\left( p \cdot q \cdot r \right) $.
	\end{enumerate}
\end{theorem}
\begin{proof}
	We will show $(1) \iff (2) \iff (3)$. Note that, $\thi{ \f{F} }{ r } = \bigcup_{r \in \overline{r}}\left( \f{F} \cdot r\right).$ However, $\overline{r} = \{r\}$ meaning that$$\thi{ \f{F} }{ r } = \f{F}\cdot r.$$
	
	We first show that $\left(1\right) \Longrightarrow \left(2\right)$. Suppose $A$ is piecewise $\f{F}$-syndetic. We need to be careful with the quantifiers in the definition of piecewise $\f{F}$-syndetic. Since $A$ is piecewise $\f{F}$-syndetic,\(\left(\exists H \in \bigtimes_{V \in \f{F}}\fs{V}\right) \ \left(\exists W \in \bigtimes_{V\in \f{F}}\f{F}\right)\)
	
	\noindent\( (\{x^{-1}\bigcup_{h \in H_V}\left( h^{-1}A\right)  \cap V: \ V \in \f{F}, \ x\in W_V\}\) has the finite intersection property).
	
	Equivalently,
	\(\left(\exists H \in \bigtimes_{V \in \f{F}}\fs{V}\right) \  \left(\exists W \in \bigtimes_{V\in \f{F}}\f{F}\right) \ \left(\exists r \in \overline{ \f{F} }\right)\)
	
	\noindent\(\left( \{x^{-1}\bigcup_{h \in H_V}\left( h^{-1}A\right) : \ V \in \f{F}, \ x\in W_V\}\subseteq r\right).\)
	
	It is important to note that the ultrafilter, $r,$ is not at all dependent on the $V \in \f{F}.$
	For all $V \in \f{F}$, $W_V \subseteq \{x: \ x^{-1}\bigcup_{h \in H_V}\left( h^{-1}A\right)  \in r\}$. Since $W_V \in \f{F},$ we may conclude that \[\{x: \ x^{-1}\bigcup_{h \in H_V}\left( h^{-1}A\right)  \in r\}\in \f{F}.\] We then have, $\left(\forall V \in \f{F}\right) \ \left(\bigcup_{h \in H_V}\left( h^{-1}A\right)  \in \f{F}\cdot r\right).$ By definition, $A \in \syn{\f{F}}{\f{F}\cdot r}.$
	
	We now show that $\left(2\right) \Longrightarrow \left(1\right)$. Suppose there is an $r \in \overline{ \f{F} }$ such that $A \in \syn{\f{F}}{\f{F}\cdot r}.$ Then for all $V \in \f{F}$ there exists an $H_V \in \fs{V}$ such that $\bigcup_{h \in H_V}h^{-1}A \in \f{F}\cdot r.$ Let $W_V =\{x: x^{-1}\bigcup_{h \in H_V}h^{-1}A \in r\}.$ Notice that $W_V \in \f{F}$. So then $\{x^{-1}\bigcup_{h \in H_V}h^{-1}A \cap V: \ V \in \f{F}, \ x\in W_V\}$ has the finite intersection property. Hence, $A$ is piecewise $\f{F}$-syndetic.

	Finally, we prove that $(2) \iff (3)$, which follows from Theorem $\ref{compcon}.$
	\begin{flalign*}
		\syn{\f{F}}{\thi{\f{F}}{r}} &= \syn{\f{F}}{\f{F} \cdot r} \\
		&= \bigcap_{s \in \overline{\f{F}\cdot r}}\bigcup_{p \in \overline{\f{F}}} \left( p\cdot s\right)  \\
		&= \bigcap_{s \in \overline{\f{F}}\cdot r}\bigcup_{p \in \overline{\f{F}}} \left( p\cdot s\right)\\
		&= \bigcap_{q \in \overline{\f{F}}} \bigcup_{p \in \overline{\f{F}}}\left( p\cdot q \cdot r\right) .
	\end{flalign*}
\end{proof}

If we add the assumption that $\overline{\f{F}}$ is a semigroup, then we get the following result.

\begin{lemma}\label{qe}
	Let $\f{F}$ be a filter on a semigroup $S$ such that $\overline{ \f{F} }$ is a subsemigroup of $\beta S$. Then, \[\bigcup_{r \in \overline{ \f{F} }}\bigcap_{ q \in \overline{ \f{F} } }\bigcup_{p \in \overline{\f{F}}}\left( p \cdot q \cdot r \right) = \bigcup_{e \in E\left(K\left(\overline{ \f{F} }\right)\right)}\bigcap_{ p \in \overline{ \f{F} } }\bigcup_{p \in \overline{\f{F}}}\left( p \cdot q \cdot e \right).\]
\end{lemma}

\begin{proof}
	 First note that, since $E\left(K\left(\overline{ \f{F} }\right)\right) \subseteq \overline{ \f{F} },$ we have \[\bigcup_{e \in E\left(K\left(\overline{ \f{F} }\right)\right)}\bigcap_{ p \in \overline{ \f{F} } }\bigcup_{p \in \overline{\f{F}}}\left( p \cdot q \cdot e \right) \subseteq \bigcup_{r \in \overline{ \f{F} }}\bigcap_{ q \in \overline{ \f{F} } }\bigcup_{p \in \overline{\f{F}}}\left( p \cdot q \cdot r \right).\]
	 
	 For the converse, consider that $\overline{ \f{F} }\cdot r$ is a left ideal of $\overline{ \f{F} }$ and hence contains a minimal left ideal, $L.$ Since $L$ is minimal, there exists a minimal idempotent of $\overline{ \f{F} }, e_r,$ such that $L=\overline{\f{F}}\cdot e_r.$ Since $\overline{ \f{F} }\cdot e_r \subseteq \overline{ \f{F} }\cdot r$, $\left(\forall u\in \overline{ \f{F} }\right)$ $\left(\exists q_u \in \overline{ \f{F} }\right)$ $\left(u\cdot e_r = q_u\cdot r\right).$ We are assuming that $\left(\exists r \in \overline{ \f{F} }\right)$ $\left( \forall q \in \overline{ \f{F} }\right)$ $\left(A \in \f{F}^* \cdot q \cdot r\right)$ and thus $\left(\forall u \in \overline{\f{F}}\right)$ $\left(A \in \f{F}^* \cdot q_u \cdot r = \f{F}^* \cdot u \cdot e_r\right)$. Hence, $$A \in \bigcup_{e \in E\left(K\left(\overline{ \f{F} }\right)\right)}\bigcap_{ q \in \overline{ \f{F} } }\left( \f{F}^* \cdot q \cdot e \right) = \bigcup_{e \in E\left(K\left(\overline{ \f{F} }\right)\right)}\bigcap_{ q \in \overline{ \f{F} } }\bigcup_{p \in \overline{\f{F}}}\left( p \cdot q \cdot e \right).$$
\end{proof}

For convenience, we define $$\ps{\f{F}}{\f{F}} := \{A: \exists B \in \syn{\f{F}}{\f{F}}, \ \exists C \in \thi{\f{F}}{\f{F}}, \ A=B \cap C\}.$$

The following is a proof of \cite[Corollary 4.5]{CJ} which is written in a way that is suggested by the above material.

\begin{theorem}
	Let $\f{F}$ be a filter on a semigroup $S$ such that $\overline{ \f{F} }$ is a subsemigroup. If $A \in \ps{\f{F}}{\f{F}}$ then $A$ is piecewise $\f{F}$-syndetic.
\end{theorem}
\begin{proof}
	Take $A \in \ps{\f{F}}{\f{F}}.$ Suppose  $B\subseteq S$ is an $\f{F}$-syndetic set and $C\subseteq S$ is an $\f{F}$-thick set such that $A = B \cap C.$ By Theorem \ref{sup}, $ \left(\forall q \in \overline{\f{F}}\right)$ $\left(\exists p_q \in\overline{\f{F}}\right)$ $\left(B \in p_q\cdot q\right)$ and $\left(\exists r \in \overline{\f{F}}\right)$ $\left(\forall q \in\overline{\f{F}}\right)$ $\left(C \in q\cdot r \right)$. Thus, $\left( \forall q \in \overline{\f{F}}\right)$ $\left(B, C \in p_{q\cdot r} \cdot q \cdot r\right)$. Being ultrafilters,  $\left(\forall q \in \overline{\f{F}}\right)$ $\left(B\cap C \in p_{q\cdot r} \cdot q \cdot r\right)$. By Theorem \ref{eq}, $A$ is piecewise $\f{F}$-syndetic.
\end{proof}

In their proof, it appears that Christopherson and Johnson implicitly assume $\overline{\f{F}\cdot q} \subseteq \overline{ \f{F} }\cdot q$. Indeed, they ultimately conclude that $\left(\overline{\f{F}} \cdot p \cdot q\right) \cap \overline{A} \ne \varnothing$ for all $p \in \overline{ \f{F} };$ however, $\syn{\f{F}}{\f{F}\cdot q} = \{A\subseteq S: \ \textrm{for every} \ r\in\overline{\f{F}\cdot q} \ \textrm{we have} \ \overline{\f{F}} \cdot r \cap \overline{A} \ne \varnothing \}.$ One can use $\overline{\f{F}\cdot q} \subseteq \overline{ \f{F} }\cdot q$ to conclude that $\{A\subseteq S: \ \textrm{for every} \ p\in\overline{\f{F}} \ \textrm{we have} \ \overline{\f{F}} \cdot p\cdot q \cap \overline{A} \ne \varnothing \} \subseteq \syn{\f{F}}{\f{F}\cdot q}.$

Since the paper of Shuungula, Zelenyuk and Zelenyuk there has been an effort to expand the theory around the algebraic structure of the Stone-\v{C}ech compactification to encompass filter relative generalizations. In \cite{BPS}, Baglini, Patra, and Shaikh prove a filter relative version of Hindman, Maleki, and Strauss' combinatorial characterization of central sets \cite{HMS}. In this paper, they prove the next two theorems. The first is a filter relative generalization of \cite[Theorem 4.39]{HiSt1}, a well known and useful theorem.
\begin{theorem}
	Let $\f{F}$ be a filter on a semigroup $S$ such that $\overline{ \f{F} }$ is a subsemigroup of $\beta S$ and let $p \in \overline{ \f{F} }.$ Then the following statement are equivalent:
	\begin{enumerate}
		\item $p \in K\left(\overline{ \f{F} }\right);$
		\item $\forall \ A\in p, \ \{x \in S: \ x^{-1}A \in p\}$ is $\f{F}$-syndetic;
		\item $\forall \ r\in \overline{\f{F}},$ $p \in \overline{ \f{F} } \cdot r \cdot p.$
	\end{enumerate}
\end{theorem}
That $(1)$ and $(2)$ above are equivalent comes from \cite[Theorem 2.2]{SZZ}. In \cite[Theorem 2.9]{BPS}, they show that $(2) \implies (3) \implies (1).$

Included in \cite[Theorem 2.10]{BPS} is the following theorem.  For the same reason that Theorem \ref{aps} may not be equivalent to \cite[Theorem 4.43]{HiSt1}, Theorem \ref{rps} is not a direct filter-relative generalization of \cite[Theorem 4.43]{HiSt1}.
\begin{theorem}\label{rps}
	Let $S$ be a semigroup. Let $A \subseteq S$ and $\f{F}$ be a filter on $S$ such that $\overline{ \f{F} }$ is a closed subsemigroup of $\beta S$. Then the following statements are equivalent:
	\begin{enumerate}
		\item $A$ is piecewise $\f{F}$-syndetic;
		\item There exists $e \in E\left(K\left(\overline{\f{F}}\right)\right)$ such that $\{x \in S: \ x^{-1}A \in e\}$ is $\f{F}$-syndetic;
		\item There exists $e \in E\left(K\left(\overline{ \f{F} }\right)\right)$ such that for every $V \in \f{F}$ there exists $x \in V$ for which $x^{-1}A \in e.$
	\end{enumerate}
\end{theorem}
We present a proof that $\left(1\right) \Longrightarrow \left(2\right)$ which is different from what can be found in \cite{BPS}.
\begin{proof}
	Suppose $A$ is piecewise $\f{F}$-syndetic. Then, by Lemma \ref{qe}, $\left(\exists e \in E\left( K\left( \overline{\f{F}}\right) \right)\right)$ $\left(\forall q \in \overline{ \f{F} }\right)$ $\left(\exists p_q \in \overline{ \f{F} } \right)$ $\left(A \in p_q \cdot q \cdot e\right).$ Thus, $\left( \forall q \in \overline{ \f{F} }\right) \left(A'\left(e\right) \in p_q \cdot q\right).$  By Theorem \ref{sup}, we may conclude that $A'\left(e\right)$ is $\f{F}$-syndetic.
\end{proof}

We use Theorem \ref{rps} to prove our main result in the same way that we used Theorem \ref{aps} to prove Theorem \ref{abc}. This answers the open question posed by Christopherson and Johnson. \cite[Question 4.6]{CJ}
\begin{theorem}\label{cpfs} Let $\f{F}$ be a filter on $S$ such that $\overline{\f{F}}$ is a subsemigroup of $\beta S$. If $A$ is piecewise $\f{F}$-syndetic, then $A \in \ps{\f{F}}{\f{F}}.$
\end{theorem}
\begin{proof}
	Let $A$ be piecewise $\f{F}$-syndetic. By Theorem \ref{rps} there is some $e \in E \left(K\left(\overline{ \f{F} }\right)\right)$ such that $A'\left(e\right)$ is $\f{F}$-syndetic. Thus, 
	\[A \cup A'\left( e\right) \ \textrm{is} \ \f{F}\textrm{-syndetic}.\]
	
	It is the case, that $\forall p \in \beta S,$ $A \in p\cdot e$ if and only if $A'\left(e\right) \in p\cdot e.$ Indeed, $A \in p\cdot e = p\cdot e \cdot e$ if and only if $A'\left(e\right) \in p\cdot e.$ Thus, $\forall p \in \beta S,$ either $A \in p\cdot e$ or $\left(A'\left(e\right)\right)^c \in p\cdot e.$ We may then conclude that, $\forall p \in \beta S$, $A \cup \left(A^c\right)'\left( e\right) \in p \cdot e.$ Thus, $A \cup \left(A^c\right)'\left( e\right)$ is thick. All thick sets are $\f{F}$-thick.
	
Finally,
	\begin{flalign*}
		\left( A \cup \left(A^c\right)'\left( e\right)\right)  \cap \left( A \cup A'\left( e\right)\right)
		= & \left( A \cup \left(A'\left( e\right)\right)^c\right)  \cap \left( A \cup A'\left( e\right)\right) \\
		=& A \cup \left( \left(A'\left( e\right)\right)^c \cap A'\left( e\right)\right) \\
		=& A \cup \varnothing = A.
	\end{flalign*}
\end{proof}

\backmatter

\bmhead{Acknowledgments}

 I would like to thank my advisor, Dr. Randall McCutcheon, for his helpful guidance. I would also like to thank the anonymous referee. Their knowledgeable comments corrected some mistakes and improved the readability of this paper.

\bibliographystyle{plain}
\bibliography{Griffin_Char_of_pFs_v3}

\end{document}